\documentclass[a4paper]{article}

\usepackage{amsfonts, amsmath, amsthm, empheq, amssymb, bm, color}

\allowdisplaybreaks[4]

\numberwithin{equation}{section}

\newtheorem{thm}{Theorem}[section]
\newtheorem{lem}[thm]{Lemma}

\newtheorem{coro}[thm]{Corollary}

\newtheorem{rem}[thm]{Remark}

\def\C{\Bbb C}
\def\N{\Bbb N}
\def\R{\Bbb R}
\def\al{\alpha}
\def\be{\beta}
\def\de{\delta}
\def\De{\Delta}
\def\ep{\epsilon}
\def\d{\mathrm d}
\def\e{\mathrm e}
\def\f{\frac}
\def\i{\mathrm i}
\def\ga{\gamma}
\def\Ga{\Gamma}

\def\te{\theta}
\def\la{\lambda}
\def\vp{\varphi}

\def\Om{\Omega}
\def\ov{\overline}
\def\pa{\partial}
\def \Re{\mathrm{\,Re\,}}

\def\wh{\widehat}

\def\B{\mathrm{B}}
\def\L{\mathcal{L}}
\def \LL{{L^2(\R^d)}}
\def\LO{{L^2(\Om)}}
\def \HH{{H^{\f{\be}{2}}(\R^d)}}

\def \HHH{{\dot{H}^{\f{\be}{2}}(\R^d)}}
\def\AA{A^{\f{\be}{2}}}

\title{\Large\bf\boldmath
Asymptotic behavior of solutions to space-time fractional diffusion equations}

\author{\large \large Xing Cheng$^1$ \qquad   Zhiyuan Li$^2$ \qquad Marahiro Yamamoto$^3$}

\date{}

\begin{document}
\maketitle

\renewcommand{\thefootnote}{\fnsymbol{footnote}}
\footnotetext{\hspace*{-5mm}
\begin{tabular}{@{}r@{}p{13cm}@{}}
& Manuscript last updated: \today.\\
$^2$ & Corresponding author\\
$^1$ & College of Science, Hohai University, Nanjing 210098,\ Jiangsu,\   China.
E-mail: chengx@hhu.edu.cn.\\
$^{2,3}$ & Graduate School of Mathematical Sciences, the University of Tokyo,
3-8-1 Komaba, Meguro-ku, Tokyo 153-8914, Japan.
E-mail: zyli@ms.u-tokyo.ac.jp, myama@ms.u-tokyo.ac.jp.
\end{tabular}}

%%%%%%%%%%%%%%%%%%%%%%%%%%%%%%%%%%%%%%%%%%%%%%%%%
%%%%%%%%%%%%%%%%%%%%%%%%%%%%%%%%%%%%%%%%%%%%%%%%%
\begin{abstract}
This article discusses the analyticity and the long-time asymptotic behavior of solutions to 
space-time fractional diffusion equations in $\R^d$. By a Laplace transform argument, 
we prove that the decay rate of the solution as $t\to\infty$ is dominated by the order of 
the time-fractional derivative. We consider the decay rate also in a bounded domain.
\end{abstract}


%%%%%%%%%%%%%%%%%%%%%%%%%%%%%%%%%%%%%%%%%%%%%%%%%
%%%%%%%%%%%%%%%%%%%%%%%%%%%%%%%%%%%%%%%%%%%%%%%%%
\section{Introduction and main results}
\label{sec-intro}
Letting $0<\al<1$ and $0<\be\le2$ be fixed constants, we first deal with the Cauchy problem
for the space-time fractional diffusion equation
\begin{equation}
\label{equ-cauchy}
\left\{
\begin{alignedat}{2}
& \pa_t^\al u(x,t) = - (-\De)^{\f{\be}{2}} u(x,t) + p(x)u(x,t),\\
& u(x,0)=a(x), \quad x\in \R^d,\ t>0.
\end{alignedat}
\right.
\end{equation}
Here we point out that the space-time fractional diffusion equation under consideration
is obtained from the standard diffusion equation by replacing the second-order space derivative
with a Riesz derivative of order $\be\in(0, 2]$ defined by
$$
(-\De)^{\f{\be}{2}}u(x,t):=\int_{\R^d} |\xi|^\be \wh{u}(\xi,t) \e^{2\pi\i\xi\cdot x} \d\xi,
$$
where $\wh{u}$ denotes the Fourier transform
$\wh{u}(\xi,t):=\int_{\R^d} u(x,t) \e^{-2\pi\i\xi\cdot x} \d x$, and the first-order time derivative with
a Caputo derivative of order $\al\in(0,1)$ that is usually defined by the formula
$$
\pa_t^\al u(x,t)
:= \f{1}{\Ga(1-\al)} \int_0^t (t-\tau)^{-\al} \f{\pa u}{\pa \tau} (x,\tau)\d \tau, \quad t>0.
$$
From the physical view point of the time-fractional and space-fractional derivatives respectively,
we see that replacing a time derivative with its fractional counterpart basically introduces
memory effects into the system and then makes the process non-Markovian,
while a replacement of a space derivative introduces non-local effects. The fundamental solution
(for the Cauchy problem) of space-time fractional diffusion equation can be interpreted
as a probability density evolving in time of a peculiar self-similar stochastic process.
See, e.g., \cite{GMMPP02}, \cite{MLP01} and the references therein.

In some recent publications such as e.g. \cite{CMN} and \cite{JT}, the space-time fractional
diffusion equations on bounded domain are investigated. The space-time fractional diffusion
equations are usually defined as
\begin{equation}
\label{equ-bdd}
\pa_t^\al u = -A^{\f{\be}{2}}u + p(x)u, \quad x \in \Om,\ t>0,
\end{equation}
where $0 < \al <1$, $0 < \be \le 2$, $\Om \subset \R^d$ is a bounded domain with 
smooth boundary $\pa\Om$ and $A := -\De$ with $D(A) = H^2(\Om) \cap H^1_0(\Om)$. 
See e.g. \cite{P92} for the definition of $A^{\f{\be}{2}}$.

For $\al=1$ and $\be=2$, the equations \eqref{equ-cauchy} and \eqref{equ-bdd} 
turn to be the classical parabolic equations. For such kind of equations we refer
Friedman \cite{F} or Pazy \cite{P92} for the long-time asymptotic estimate of
the solutions. We also refer Yamamoto \cite{Y09} where several kinds of
Carleman estimates are established and applied to some inverse problems.

For $\al=1$ and $\be\in(0,2)$, it becomes the space fractional diffusion equation.
See, e.g., \cite{GM98}, \cite{MPR07} and the references therein.

For $\al\in(0,1)$ and $\be=2$, it is the so-called time-fractional diffusion equation.
For such a kind of equations there are a large and rapidly growing number of publications.
In \cite{LiLiuYa14}, the well-posedness and the long-time asymptotic behavior for
initial-boundary value problems for multi-term time-fractional diffusion equations with
positive constant coefficients were studied by exploiting several important properties of
multinomial Mittag-Leffler functions. \cite{LiYa14} generalized the results in \cite{LiLiuYa14}
by considering the counterpart with spatial-variable positive coefficients. 
In \cite{LiLuYa15}, the asymptotic behavior of the solutions to a more general 
time-fractional diffusion equations with distributed order derivatives were investigated.
For other results for the time-fractional diffusion equations, we refer to \cite{Lu09} for 
a maximum principle,  \cite{Lu11}, \cite{SaYa11} for the initial-boundary value problems.

In this paper, we are devoted to the long time asymptotic behavior
for the space-time fractional diffusion equation. 
To the best knowledge of the authors, a long time asymptotic estimate for
space-time fractional diffusion equation with nonpositive potential is not yet established.

For later use, we introduce some notations.
We first define the homogeneous Sobolev space
$\dot H^\ga(\R^d)$ with $\ga\in \mathbb{R}$ to be
$$
\dot H^\ga(\R^d) :=\big\{ u\in \mathcal{S}':  \wh u(0) = 0, \ |\xi|^\ga \wh u(\xi) \in \LL \big\}
$$
equipped the norm $\|u\|_{\dot H^\ga(\R^d)}:=\big\||\xi|^\ga \wh u(\xi)\big\|_\LL$,
and define the non-homogeneous Sobolev space $H^\ga(\R^d)$ to be
$$
H^\ga(\R^d) :=\big\{ u\in \mathcal{S}': \ (1+|\xi|^2)^{\f{\ga}{2}} \wh u(\xi) \in \LL \big\}
$$
with the norm $\|u\|_{H^\ga(\R^d)}:=\big\|(1+|\xi|^2)^{\f{\ga}{2}} \wh u(\xi)\big\|_\LL$.

\vspace{0.3cm}

We are ready to state our first main result.
\begin{thm}
\label{thm-asymp}
Let $0<\al<1$ and $0<\be\le2$ be given constants, $a\in L^1(\R^d)\cap\LL$,
$p\in L^\infty(\R^d)$ with $p(x)\le0$ in $\R^d$. Then
$$
\|u(\cdot,t)\|_\HHH \le
\left\{
\begin{alignedat}{2}
&Ct^{-\f{\al}{2}} \|a\|_\LL,&\quad d\le\be\\
&Ct^{-\al}(\|a\|_{L^1(\R^d)}+\|a\|_\LL),&\quad d>\be
\end{alignedat}
\right.,\quad t>0
$$
is valid for the solution $u$ to the Cauchy problem \eqref{equ-cauchy}.

Moreover, if the potential $p$ satisfies $p<-\de_0$, where $\de_0$ is a positive constant.
Then there exists $v(\cdot,t)\in H^{\f{\be}{2}}(\R^d)$ be the solution of the elliptic problem
\begin{equation}
\label{equ-ellip-v}
(-\De)^{\f{\be}{2}} v(x,t) - p(x)v(x,t) = \f{t^{-\al}}{\Ga(1-\al)}a(x), \quad x\in\R^d,\ t>0,
\end{equation}
which has the same asymptotic behavior as $u$, in the sense that
\begin{equation}
\label{esti-asymp-u}
\|u(\cdot,t) - v(\cdot,t)\|_\HH =O(t^{-2\al})\|a\|_\LL,\mbox{ as $t\to\infty$}.
\end{equation}
\end{thm}
In the recent work \cite{KSZ15} and \cite{KL15}, by the use of Fourier analysis techniques,
the asymptotic estimates of solutions to space-time fractional diffusion equations with
potential $p=0$ were studied.  Their methods are based on the explicit representation
formula of solutions, and thus cannot be used for the space-time fractional diffusion equations
with non-zero potential.

The above theorem shows that the decay rate of $u$ is $t^{-\al}$ at best
provided $p<-\de_0<0$. More precisely, we have the following statement:
\begin{coro}
\label{coro-asymp}
Let $0<\al<1$ and $0<\be\le2$ be given constants, $a\in L^1(\R^d)\cap\LL$,
$p\in L^\infty(\R^d)$ with $p(x)<-\de_0$ in $\R^d$, where $\de_0$ is a positive constant.
Moreover we suppose that
$$
\|u(\cdot,t)\|_\HH =o(t^{-\al}),\mbox{ as $t\to\infty$.}
$$
Then $u(x, t) = 0$ for all $x\in\R^d$ and $t > 0$.
\end{coro}

Now we turn to considering the asymptotic behavior of solutions to the space-time fractional 
diffusion equations in a bounded domain. Actually, we can consider a more general symmetric 
elliptic operator of second order such that $A$ is positive, that is, there exists a constant
$\de_0>0$ satisfying $(Au,u) \ge \de_0\| u\|_{L^2(\Om)}^2$ for all
$u \in D(A)$.  We consider the following initial-boundary value problem
\begin{equation}
\label{equ-gov-bdd}
\left\{ \begin{alignedat}{2}
& \pa_t^\al u = -A^{\f{\be}{2}}u + p(x)u, &\quad &x \in \Om, \ t>0, \\
& u(x,0) = a(x), &\quad &x \in \Om,\\
& u(x,t)= 0, &\quad& x\in\pa\Om,\ t>0.
\end{alignedat}
\right.
\end{equation}

We have
\begin{thm}
\label{thm-asymp-bdd}
Let $0<\al<1$, $0<\be\le2$ and $0\le\ga\le\be$ be given constants, $a\in L^2(\Om)$,
$p\in L^\infty(\Om)$ with $p(x)\le0$ in $\Om$. Then
\begin{equation}
\label{esti-u-gamma}
\|u(\cdot,t)\|_{H^\ga(\Om)} \le Ct^{-\al} \|a\|_{L^2(\Om)}
,\quad t>0
\end{equation}
is valid for the solution $u$ to the initial-boundary value problem \eqref{equ-gov-bdd}.
\end{thm}

\begin{rem}
For $a \in L^2(\Omega)$, as is proved in Section 3, in general we have
\begin{equation}
\label{notin}
u(t) \not\in H^{\ga}(\Om) \quad \mbox{if $\ga > \be$}.
\end{equation}
Moreover the asymptotic behavior in time is determined only by the order
of the fractional $t$-derivative.
As \eqref{esti-u-gamma} shows, the order of fractional $x$-derivative influences only the
spatial regularity.
\end{rem}
The rest of this paper is organized in three sections. In Section \ref{sec-asymp},
we investigate the analyticity of the solution to the Cauchy problem \eqref{equ-cauchy}.
By the Fourier-Laplace argument, we obtain an integral equation of the solution,
and based on this, we further show that the solution can be analytically extended to the sector
$S_{\te_0}:=\{z\in\C\setminus\{0\}; |\arg z|<\te_0\}$ ($\f{\pi}{2}<\te_0<\min\{\pi,\f{\pi}{2\al}\}$).
In Section \ref{sec-proof}, we complete the proof of Theorem 1.1.
In Section \ref{sec-asymp-bdd}, we apply 
the eigenfunction expansion and the property of the Mittag-Leffler function to derive 
the asymptotic behavior of the solution to the initial-boundary value problem \eqref{equ-gov-bdd}. 
Finally, the concluding remarks are given in Section \ref{sec-rem}.

%%%%%%%%%%%%%%%%%%%%%%%%%%%%%%%%%%%%%%%%%%%%%%%%%
%%%%%%%%%%%%%%%%%%%%%%%%%%%%%%%%%%%%%%%%%%%%%%%%%
\section{Asymptotic estimate in $\R^d$}
In this section, we mainly investigate the long-time asymptotic estimate of the solution to
the Cauchy problem \eqref{equ-cauchy}. Our proof is divided into two parts. In part 1, 
we will give an integral equation which is equivalent to \eqref{equ-cauchy}. 
Based on this integral equation, we show the analyticity of the solution. In part 2, 
according to part 1 and the Laplace transform argument, the long-time asymptotic
behavior is easily proved.

\subsection{Analyticity of the solution}
\label{sec-asymp}
In this subsection, we are to give some properties of the solution $u$
to the Cauchy problem \eqref{equ-cauchy}, such as the continuously
depending on the initial data, the analyticity, etc., which are useful
for the proof of Theorem \ref{thm-asymp}. For this, we start from the
following observations and notations. We first claim that the function
\begin{equation}
\label{def-S}
w(x,t)
= \int_{\R^d} E_{\al,1}(-|\xi|^\be t^\al) \wh{w_0}(\xi)\e^{2\pi\i\xi\cdot x}  \d\xi
=:S(t)w_0,\quad x\in\R^d,\ t>0,
\end{equation}
where $E_{\al,1}(\cdot)$ is the Mittag-Leffler function, see, e.g., \cite{P99},
solves the Cauchy problem
\begin{equation}
\label{equ-homo}
\left\{
\begin{alignedat}{2}
& \pa_t^\al w(x,t) = - (-\De)^{\f{\be}{2}} w(x,t),&\quad &x\in \R^d,\ 0<t\le T, \\
& w(x,0)=w_0(x), &\quad &x\in \R^d.
\end{alignedat}
\right.
\end{equation}
In fact, we can follow the Laplace-Fourier transform used in \cite{MLP01} to
derive the above statement. By a direct calculation, we find
\begin{equation}
\label{equ-S'}
S'(t)a(x)
= -\int_{\R^d} |\xi|^\be z^{\al-1} E_{\al,\al}(-|\xi|^\be z^\al) \hat{a}(\xi) \e^{2\pi\i\xi\cdot x} \d\xi.
\end{equation}
Based on the formula \eqref{def-S} of the solution
to the Cauchy problem \eqref{equ-homo}, by regarding the term $p(x)u(x,t)$
as a source term, we find an implicit form of the solution $u$ to the Cauchy problem
\eqref{equ-cauchy} for $t\in(0,T)$, say, the following integral equation of $u$,
\begin{equation}
\label{equ-int-u}
u(t) = S(t)a -\int_0^t (-\De)^{-\f{\be}{2}} S'(t-\tau) (pu(\tau)) \d\tau,
\end{equation}
here we omit the spatial variable $x$, and we call a function $u$ satisfies 
the above integral equation as the mild solution to the Cauchy problem \eqref{equ-cauchy}.

\begin{lem}
\label{lem-analy-u}
Let $0<\al<1$, $0<\be\le2$ and $0\le\ga<\be$ be given constants, $a\in L^1(\R^d)\cap\LL$
and $p\in L^\infty(\R^d)$. There exists a unique mild solution $u\in H^\ga(\R^d)$ to
the Cauchy problem \eqref{equ-cauchy}. Moreover, the solution $u:(0,T)\to H^\ga(\R^d)$
is analytic and can be analytically extended to the sector $S_{\te_0}$,
and the following estimate
\begin{equation}
\label{esti-u}
\|u(\cdot,z)\|_{H^\ga(\R^d)} \le C\e^{CT}|z|^{-\f{\ga}{\be}\al}\|a\|_\LL,\quad 0\le\ga<\be,
\end{equation}
is valid on a sector $S_{\te_0,T}:=\{z\in\C\setminus{\{0\}}: |\arg z|<\te_0,\ |z|\le T\}$,
where $\f{\pi}{2}<\te_0<\min\{\pi,\f{\pi}{2\al}\}$ and the constant $C>0$ is independent of
$a$, $z$, but may depend on $d$, $\al$, $\be$, $\ga$, $p$, $\te_0$.
\end{lem}
\begin{proof}
After a change of variables in the integral equation \eqref{equ-int-u}, we find
$$
u(t) = S(t)a - t \int_0^1 (-\De)^{-\f{\be}{2}} S'(t(1-\tau)) (pu(t\tau)) \d\tau.
$$
We now extend the variable $t$ to the sector $S_{\te_0}$, and setting
$u_0=0$, we define $u_{n+1}$ ($n=0,1,\cdots$) as follows:
\begin{equation}
\label{def-u_n}
u_{n+1}(z) = S(z)a - z \int_0^1 (-\De)^{-\f{\be}{2}} S'(z(1-\tau)) (pu_n(z\tau))\d\tau.
\end{equation}

We first claim that
\begin{equation}
\label{esti-induc-u1}
\|u_{n+1}(\cdot,z) - u_n(\cdot,z)\|_\LL
\le C\f{\Ga(\al)^n}{\Ga(n\al+1)} |z|^{n\al} \|a\|_\LL,\quad z\in S_{\te_0}.
\end{equation}
In fact, first for $n=0$, noting that the Mittag-Leffler function $E_{\al,\rho}(-z)$
($\al>0$, $\rho\in\R$) is bounded on the sector $S_{\te_0}$, see, e.g., \cite{P99}
and \cite{SaYa11}, we conclude from the Plancherel identity that the following estimates
$$
\|u_1(\cdot,z) - u_0(\cdot,z)\|_\LL
=\|S(z)a\|_\LL
=\|E_{\al,1}(-|\xi|^\be z^\al) \hat{a}(\xi)\|_\LL
\le C\|a\|_\LL
$$
are valid on the sector $S_{\te_0}$. Now we assume that the statement
\eqref{esti-induc-u1} is valid for $n-1\in\N$, that is,
$$
\|u_n(\cdot,z) - u_{n-1}(\cdot,z)\|_\LL
\le C\f{\Ga(\al)^{n-1}}{\Ga((n-1)\al+1)} |z|^{(n-1)\al} \|a\|_\LL,\quad z\in S_{\te_0}.
$$
For $n\in\N$, again that the Plancherel identity and the boundedness of 
the Mittag-Leffler function used in the evaluation for $n=0$ yield
\begin{align*}
&\|u_{n+1}(z)- u_n(z)\|_\LL
= \left\|z\int_0^1 (-\De)^{-\f{\be}{2}} S'(z(1-\tau)) p(u_n - u_{n-1})(z\tau) \d\tau\right\|_\LL
\\
= \ &|z|
\left\|
\int_0^1 (z(1-\tau))^{\al-1} E_{\al,\al}(-|\xi|^\be (z(1-\tau))^\al) (p(u_n - u_{n-1})(z\tau))^\wedge \d\tau
\right\|_\LL
\\
\le&  C|z|^\al
\int_0^1 (1-\tau)^{\al-1} \|(u_n - u_{n-1})(z\tau)\|_\LL \d\tau.
\end{align*}
By the inductive assumption, we see that
$$
\|u_{n+1}(z)- u_n(z)\|_\LL
\le C\f{\Ga(\al)^{n-1}}{\Ga((n-1)\al + 1)} \f{\Ga(\al)\Ga((n-1)\al+1)}{\Ga(n\al+1)}|z|^{n\al}\|a\|_\LL,
$$
which proves the statement \eqref{esti-induc-u1}.

Now taking the operator $(-\De)^{\f{\ga}{2}}$ ($0\le\ga<\be $) on both sides of
\eqref{def-u_n}, for $n=0$, by noting the definition of $u_n$ and the Plancherel identity,
we derive
$$
\|(-\De)^{\f{\ga}{2}}(u_{1}(z) - u_0(z))\|_\LL
=\|(-\De)^{\f{\ga}{2}}S(z)a\|_\LL
=\left\||\xi|^\ga \wh{S(z)a}(\xi)\right\|_\LL.
$$
From the definition of the operator $S(t)$ and the asymptotic expansion of
the Mittag-Leffler function, see, e.g., Theorem 1.6 in \cite{P99}, it follows that
$$
\|(-\De)^{\f{\ga}{2}}(u_{1}(z) - u_0(z))\|_\LL
\le C \left\| \f{|\xi|^\ga}{1+|\xi|^\be |z|^\al} |\hat{a}(\xi)|\right\|_\LL.
$$
Moreover, by noting that
\begin{equation}
\label{esti-xi}
\f{|\xi|^\ga |z|^{\f{\ga}{\be}\al}}{ 1 + |\xi|^\be |z|^\al }
\le C \mbox{ with $\ga\le\be$},
\end{equation}
we have
$$
\|(-\De)^{\f{\ga}{2}}(u_{1}(z) - u_0(z))\|_\LL
\le\left\| \f{|\xi|^\ga |z|^{\f{\ga}{\be}\al}}{ 1+  |\xi|^\be |z|^\al } |\hat{a}(\xi)|  |z|^{-\f{\ga}{\be}\al}\right\|_\LL
\le C|z|^{-\f{\ga}{\be}\al}\|a\|_\LL.
$$
For $n=1,2,\cdots$, in view of the formula \eqref{equ-S'} and the asymptotic estimate of
the Mittag-Leffler function, it follows that
\begin{align*}
&\|(-\De)^{\f{\ga}{2}}(u_{n+1}(z)- u_n(z))\|_\LL
=  \left\|z\int_0^1 (-\De)^{\f{\ga-\be}{2}} S'(z(1-\tau)) p(u_n - u_{n-1})(z\tau) \d\tau\right\|_\LL
\\
\le \ & C|z|
\left\|
\int_0^1 \f{|\xi|^\ga (|z|(1-\tau))^{\al-1}}{ 1+ |\xi|^\be (|z|(1-\tau))^\al }
    \left |(p(u_n - u_{n-1})(z\tau))^\wedge\right| \d\tau
\right\|_\LL.
\end{align*}
Again the use of the inequality \eqref{esti-xi} and the assumption $p\in L^\infty(\R^d)$ imply
\begin{align*}
\|(-\De)^{\f{\ga}{2}}(u_{n+1}(z)- u_n(z))\|_\LL
\le C |z|^{\al(1-\f{\ga}{\be})}
\int_0^1 (1-\tau)^{\al(1-\f{\ga}{\be})-1} \|(u_n - u_{n-1})(z\tau)\|_\LL \d\tau
\end{align*}
for $z\in S_{\te_0}$. Consequently, substituting \eqref{esti-induc-u1} into above inequality,
we can prove
\begin{align*}
\|(-\De)^{\f{\ga}{2}}(u_{n+1}(z)- u_n(z))\|_\LL
\le  \f{C\Ga(\al)^n |z|^{\al(n+1-\f{\ga}{\be})}}{\Ga(n\al+1)} \int_0^1(1-\tau)^{\al(1-\f{\ga}{\be})-1}\tau^{n\al} \d \tau
\|a\|_\LL.
\end{align*}
Hence
$$
\|(-\De)^{\f{\ga}{2}}(u_{n+1}(z)- u_n(z))\|_\LL
\le\f{C\Ga(\al)^n \Ga(\al(1-\f{\ga}{\be}))}{\Ga( \al(n+1-\f{\ga}{\be}) +1 )}  |z|^{\al(n+1-\f{\ga}{\be})}\|a\|_\LL.
$$

It is not difficult to see from the inductive argument that for each $n\in\N$,
$u_n(\cdot,z)$ is analytic in sector $S_{\te_0}$ under the norm $H^{\ga}(\R^d)$
with $0<\ga<\be$. Moreover, the classical asymptotics
$$
\Ga(\eta) = e^{-\eta}\eta^{\eta-\f 1 2} (2\pi)^{\f 1 2}
\left(1 + O\left(\f{1}{\eta}\right)\right) \quad \mbox{as}\quad \eta
\rightarrow +\infty 
$$
(e.g., Abramowitz and Stegun \cite{AS72}, p.257) imply that the series
$$
\sum_{n=0}^\infty \|(-\De)^{\f{\ga}{2}}(u_{n+1}(\cdot,z) - u_n(\cdot,z))\|_\LL
$$
and
$$
\sum_{n=0}^\infty \|u_{n+1}(\cdot,z) - u_n(\cdot,z)\|_\LL
$$
converge uniformly in any compact subset of $S_{\te_0}$, therefore
$$
u(\cdot,z):=\sum_{n=0}^\infty (u_{n+1}(\cdot,z) - u_n(\cdot,z)):S_{\te_0}\to H^\ga(\R^d)
$$
is analytic. Finally, we see that $u(\cdot,t)$ ($t\in(0,T]$) is just the solution to
the Cauchy problem
\eqref{equ-cauchy}.

From the integral equation \eqref{equ-int-u} of $u$, similar to the above argument,
we see that
$$
\|u(\cdot,z)\|_\LL
\le C\|a\|_\LL + C\int_0^{|z|} (|z|-\tau)^{\al-1} \|u(\cdot,z)\|_\LL \d\tau.
$$
Thus by the generalized 
Gronwall inequality (see e.g. Lemma 7.1.1 in \cite{H81}), we find
$$
\|u(\cdot,z)\|_\LL
\le C\e^{CT}\|a\|_\LL,\quad z\in S_{\te_0,T},
$$
combining with the integral equation \eqref{equ-int-u}, hence that
$$
\|(-\De)^{\f{\ga}{2}}u(\cdot,z)\|_\LL
\le C\e^{CT}|z|^{-\f{\ga}{\be}\al} \|a\|_\LL,
\quad z\in S_{\te_0,T}.
$$
This completes the proof of Lemma \ref{lem-analy-u}.
\end{proof}

%%%%%%%%%%%%%%%%%%%%%%%%%%%%%%%%%%%%%%%%%%%%%%%%%
%%%%%%%%%%%%%%%%%%%%%%%%%%%%%%%%%%%%%%%%%%%%%%%%%
\subsection{Proof of Theorem \ref{thm-asymp}}
\label{sec-proof}
In this part, we are to finish the proof for the asymptotic behavior of
the solution $u$ to the Cauchy problem \eqref{equ-cauchy}.
We first investigate the property of the Laplace transform $\L u$ of the solution
$u$. For this, we apply the Laplace transform
$$
(\L f)(s):=\int_0^\infty f(t) \e^{-st} \d t
$$
on both sides of the differential equation in \eqref{equ-cauchy} and use the formula
$$
\L (\pa_t^\al f) (s)=s^\al (\L f)(s)-s^{\al-1}f(0+),
$$
to derive the transformed algebraic equation
$$
s^\al \L u(\cdot\,;s) = - (-\De)^{\f{\be}{2}} \L u(\cdot\,;s) + p(x)\L u(\cdot\,;s)+s^{\al-1} a,\ 
s\in \C, \  \Re s>M,
$$
where $M>0$ is a sufficiently large constant.

We check at once that $\L u:\{s\in \mathbb{C}: \Re s>M\}\to \HH$ is analytic,
which is clear from the property of the Laplace
transform. Moreover, we claim that $\L u(x;s)$ ($\Re s>M$)
can be analytically extended to the sector
$\{s\in\C\setminus{\{0\}}: |\arg s | < \pi\}$, that is,
\begin{lem}
\label{lem-analy_exten}
Under the assumptions in Lemma \ref{lem-analy-u}, the Laplace transform $\L u$ of
the unique solution $u$ to the Cauchy problem \eqref{equ-cauchy} can be analytically
extended to the sector $\{s\in\C\setminus{\{0\}}:  |\arg s| < \pi\}$.
\end{lem}
\begin{proof}
By Lemma \ref{lem-analy-u}, it follows that $u(z)$ is holomorphic on the sector
$|\arg z| < \te\le\f{\pi}{2}$ under the norm in $H^\ga(\R^d)$ ($0\le\ga<\be$), such that \eqref{esti-u},
we define $f(s)=\L u(s),s>0$. For $\phi<\te$, consider the path
$\ga_{R,r} = r\e^{\i[0,\phi]},[r,R],R\e^{\i[0,\phi]}, [R,r]\e^{\i\phi}$;
by the Cauchy theorem $\int_{\ga_{R,r}} u(z)\e^{-s z} \d z =0$, i.e.
with $r\to0,R\to\infty$ we obtain
$$
f(s) = \int_0^\infty u(t) \e^{-s t} \d t
=\int_0^\infty u(t\e^{\i\phi}) \e^{-s t\e^{\i\phi}} \e^{\i\phi} \d t,\quad s>0.
$$
Since $\phi\in(-\te,\te)$ can be arbitrarily chosen, it follows that $f(s)$
can be analytically extended to the sector
$\{s\in\C\setminus{\{0\}}:  |\arg s|<\te+\f{\pi}{2}\}$ ($\forall\, \te\in(0,\f{\pi}{2})$).
This complete the proof of Lemma \ref{lem-analy_exten}.
\end{proof}
See e.g., Theorem 0.1, p.5 in \cite{P93} for a similar argument.
Moreover, we have the following lemma
about the evaluation of the Laplace transform $\L u$ on the sector
$S_{\te_0}$.
\begin{lem}\label{lem-esti-Lu}
Let $0<\al<1$ and $0<\be\le2$ be given constants, $a\in L^1(\R^d)\cap\LL$,
$p\in L^\infty(\R^d)$ with $p(x)\le0$ in $\R^d$. There exists a constant $C$
depending on $d$, $\te_0$, $\al$, $\be$ such that
\begin{equation}
\label{esti-lap-u}
\|\L u(\cdot\,;s)\|_\HHH
\le
\left\{
\begin{alignedat}{2}
&C|s|^{\f{\al}{2}-1} \|a\|_\LL,&\quad d\le\be\\
&C|s|^{\al-1}(\|a\|_{L^1(\R^d)}+\|a\|_\LL),&\quad d>\be
\end{alignedat}
\right.,\quad\forall \, s\in S_{\te_0}.
\end{equation}
\end{lem}
\begin{proof}
From the above argument, we see that
\begin{equation}
\label{equ-Lu}
(-\De)^{\f{\be}{2}} \L u(\cdot\,;s) + (s^\al - p(x)) \L u(\cdot\,;s) =s^{\al-1} a,
\quad \ s\in S_{\te_0}.
\end{equation}
Firstly for any fixed $s=r\e^{\i\te}\in S_{\te_0}$, we define a bilinear operator
$\B[\Phi,\Psi;s]: \HH\times \HH\to\C$ by
$$
\B[\Phi,\Psi;s]
:=\int_{\R^d} \left((-\De)^{\f{\be}{2}}\Phi(x)\right) \ov{\Psi(x)}
+ (s^\al - p(x)) \Phi(x) \ov{\Psi(x)} \d x,
$$
where $\ov{\Psi}$ denotes conjugate of $\Psi$. Consequently,
for any $s\in S_{\te_0}$, \eqref{equ-Lu} can be written as
$$
\B[\L u,\Psi;s] = (s^{\al-1} a,\Psi)_\LL, \quad \forall\,  \Psi\in \HH, \ s\in S_{\te_0}.
$$
Furthermore, taking $\Psi=\L u$ in the above equality,
from the Parseval identity and Fourier inversion transform, we have
\begin{align*}
\B[\L u,\L u;s]
&=\int_{\R^d} \left( (-\De)^{\f{\be}{2}}\L u(\xi;s) \right)^\wedge \ov{\wh{\L u}(\xi;s)} \d \xi
+ \int_{\R^d} (s^\al- p(x)) \L u(x;s) \ov{\L u(x;s)} \d x
\\
&=\int_{\R^d} \left(|\xi|^{\f{\be}{2}} \wh{\L u}(\xi;s) \right) \left( \ov{|\xi|^{\f{\be}{2}}\wh{\L u}(\xi;s)} \right)\d \xi
+ \int_{\R^d} (s^\al- p(x)) |\L u(x;s)|^2 \d x
\\
&=\int_{\R^d} (-\De)^{\f{\be}{4}} \L u(x;s) \ov{(-\De)^{\f{\be}{4}} \L u(x;s)} \d x
+ \int_{\R^d} (s^\al- p(x)) |\L u(x;s)|^2 \d x.
\end{align*}
In view of $\cos(\al\te)>0$ for $0<\te<\te_0$ and the assumption $p\le0$, it follows that
\begin{equation}
\label{esti-ReLu}
\Re \B[\L u,\L u;s]
\ge \left\|(-\De)^{\f{\be}{4}} \L u(\cdot\,;s) \right\|_\LL^2  + C|s|^\al\|\L u(\cdot\,;s)\|_\LL.
\end{equation}
In the case of $d>\be$, we conclude from \eqref{esti-ReLu} that
$$
\|\L u(\cdot\,;s)\|_{\dot{H}^\f{\be}{2}(\R^d)}^2
\le C\big|\B[\L u,\L u;s]\big|
\le C|s|^{\al-1}|(a, \L u(\cdot\,;s))_\LL|.
$$
We note that
$$
(a, \L u(\cdot\,;s))_\LL
=\int_{\R^d} a(x) \ov{\L u(x;s)} \d x
=\int_{\R^d}\left( |\xi|^{-\f{\be}{2}}\wh a(\xi) \right) \left(|\xi|^{\f{\be}{2}}\ov{\wh{\L u}(\xi;s)}\right) \d\xi,
$$
from the H\"older inequality, we see that
$$
\|\L u(\cdot\,;s)\|_{\dot{H}^\f{\be}{2}(\R^d)}^2
\le C|s|^{\al-1} \left(\int_{\R^d} |\xi|^{-\be} |\wh a(\xi)|^2 \d\xi \right)^{\f 1 2}
\left(\int_{\R^d}|\xi|^{\be} \left|\wh{\L u}(\xi;s)\right|^2\d\xi \right)^{\f 1 2},
$$
which implies
$$
\|\L u(\cdot\,;s)\|_{\dot{H}^\f{\be}{2}(\R^d)}
\le C|s|^{\al-1} \left(\int_{\R^d} |\xi|^{-\be} |\wh a(\xi)|^2 \d\xi \right)^{\f 1 2},
\quad s\in S_{\te_0}.
$$
Breaking the integral on the right-hand side of the above estimate into $\{|\xi|<1\}$ and $\{|\xi|\ge1\}$, in view of
$\|\wh a\|_{L^\infty(\R^d)}\le \|a\|_{L^1(\R^d)}$, it follows that
$$
\int_{\R^d} |\xi|^{-\be} |\wh a(\xi)|^2 \d\xi
\le \|a\|_{L^1(\R^d)}^2 \int_{|\xi|<1} |\xi|^{-\be} \d\xi
+\int_{|\xi|\ge1} |\wh a(\xi)|^2\d\xi
\le C(\|a\|_{L^1(\R^d)}^2+\|a\|_\LL^2),
$$
hence that
$$
\|\L u(\cdot\,;s)\|_{\dot{H}^\f{\be}{2}(\R^d)}
\le C|s|^{\al-1} (\|a\|_{L^1(\R^d)}+\|a\|_\LL),
\quad s\in S_{\te_0},\ d>\be.
$$
Next for $d\le\be$, noting that
$\|a\|_\LL\|\L u\|_\LL\le \f{1}{4\ep_0}|s|^{-1}\|a\|_\LL^2 +\ep_0|s|\|\L u\|_\LL^2$ with sufficiently small
constant $\ep_0>0$, we see from \eqref{esti-ReLu} that
$$
\|\L u(\cdot\,;s)\|_\HHH
\le C|s|^{\f{\al}{2}-1} \|a\|_\LL,\quad s\in S_{\te_0},\ d\le\be.
$$
The proof of Lemma \ref{lem-esti-Lu} is completed.
\end{proof}

Now let us turn to finishing the proof of the long-time asymptotic behavior of
the solution $u$ to the Cauchy problem \eqref{equ-cauchy}.
\begin{proof}[\bf Proof of Theorem \ref{thm-asymp}]
By the Fourier-Mellin formula (e.g., \cite{S91}), we have
$$
u(x,t) = \frac{1}{2\pi \i}\int_{M-\i\infty}^{M+\i\infty} \L u(x;s) \e^{st} \d s.
$$
From Lemma \ref{lem-esti-Lu}, we see that the Laplace transform $\L u(x;s)$ of
the solution to the equation \eqref{equ-Lu} is analytic in the sector $S_{\te_0}$.
Therefore by the Residue Theorem (e.g., \cite{W86}), for $t>0$ we see that
the inverse Laplace transform of $\L u$ can be represented by an integral
on the contour $\ga(\ep,\te_0)$ defined as
$\{s\in\C: \arg s=\te_0,\ |s|\ge \ep\}\cup \{s\in\C: |\arg s|\le\te_0,\ |s|= \ep\}$, that is
$$
u(x,t)
=\frac{1}{2\pi\i}\int_{\ga(\ep,\te_0)} \L u(x;s) \e^{s t} \d s,
$$
where the shift in the line of integration is justified by the estimate \eqref{esti-lap-u}.
Moreover, again from the estimate \eqref{esti-lap-u}, we can let $\ep$ tend to $0$,
then we have
$$
u(x,t)=\frac{1}{2\pi\i}\int_{\ga(0,\te_0)} \L u(x;s)\e^{s t} \d s.
$$
Thus
\begin{equation}\label{esti-u-v}
\|u(\cdot\,,t)\|_\HHH
\le C\int_{\ga(0,\te_0)}\|\L u(\cdot\,;s)\|_\HHH |\e^{st} \d s|,
\end{equation}
hence combining \eqref{esti-lap-u} with \eqref{esti-u-v} gives
$$
\|u(\cdot\,,t)\|_\HHH
\le C\int_0^\infty r^{\f{\al}{2}-1} \|a\|_\LL \e^{rt\cos\te_0} \d r
\le Ct^{-\f{\al}{2}}\|a\|_\LL, \quad t>0,\ d\le\be,
$$
and
$$
\|u(\cdot\,,t)\|_\HHH
\le C\int_0^\infty r^{\al-1} \|a\|_\LL \e^{rt\cos\te_0} \d r
\le Ct^{-\al}(\|a\|_{L^1(\R^d)}+\|a\|_\LL), \quad t>0,\ d>\be.
$$
Now let us turn to consider the asymptotic behavior of the solution $u$
in the case of $p<-\de_0$, where $\de_0$ is a positive constant.
We repeat the above argument to derive that the Laplace transform
$\L v$ of the solution $v$ to the equation \eqref{equ-ellip-v} such that
\begin{equation}
\label{equ-Lv}
(-\De)^{\f{\be}{2}} \L v(\cdot\,;s) - p(x) \L v(\cdot\,;s) = s^{\al-1} a,\quad s\in S_{\te_0}.
\end{equation}
We define a bilinear operator $\B_1[\Phi,\Psi]: \HH\times \HH\to\C$
corresponding the equation \eqref{equ-Lv} by
$$
\B_1[\Phi,\Psi]
:=\int_{\R^d} \left((-\De)^{\f{\be}{2}}\Phi(x)\right) \ov{\Psi(x)} - p(x) \Phi(x) \ov{\Psi(x)} \d x.
$$
By noting $p(x)<-\de_0<0$ in $\R^d$, similar to the proof used in Lemma \ref{lem-esti-Lu},
we find
$$
\|\L v(\cdot\,;s)\|_{\HH}^2
\le \Re \B_1[\L v,\L v]
\le C|s|^{\al-1}\|a\|_{\LL}\|\L v(\cdot;s)\|_\LL, \quad s\in S_{\te_0}.
$$

Noting that $\L u - \L v$ satisfies the following problem
\begin{equation*}
(-\De)^{\f{\be}{2}} (\L u(\cdot\,;s)-\L v(\cdot\,;s)) +(s^\al- p(x))(\L u(\cdot\,;s) - \L v(\cdot\,;s))
= s^\al \L v(\cdot\,;s),\quad s\in S_{\te_0}.
\end{equation*}
Then again by the argument used in the evaluation of \eqref{esti-lap-u}, we deduce that
$$
\|\L u(\cdot\,;s) - \L v(\cdot\,;s)\|_\HH
\le C|s|^\al \|\L v(\cdot\,;s)\|_\LL.
$$
Consequently
\begin{align*}
\|u(\cdot\,,t) - v(\cdot,t)\|_\HH
\le \ & C\int_{\ga(0,\te_0)}\|\L (u(\cdot\,;s) - v(\cdot;s))\|_\HH |\e^{st} \d s|\\
\le\ & C\int_0^\infty r^{2\al-1} \|a\|_\LL \e^{rt\cos\te_0} \d r
\le Ct^{-2\al}\|a\|_\LL, \quad t>0.
\end{align*}
This completes the proof of Theorem \ref{thm-asymp}.
\end{proof}

\begin{proof}[\bf Proof of Corollary \ref{coro-asymp}]
The relation
$$
\|t^\al u(\cdot,t) - t^\al v(\cdot,t)\|_\HH = t^\al O(t^{-2\al})\|a\|_\LL \to 0,\ \ t\to\infty
$$
follows easily from the estimate \eqref{esti-asymp-u} in Theorem~\ref{thm-asymp}.
By the assumption of Corollary \ref{coro-asymp}, $\|t^\al u(\cdot,t)\|_\HH\to 0$
as $t \to \infty$, which implies the relation $t^\al v(\cdot,t) \to 0$ as $t\to\infty$ under the norm
$\|\cdot\|_\HH$. By noting that $v$ solves the equation \eqref{equ-ellip-v}, it follows that
$$
B_1[ v(\cdot,t),\vp]=\f{t^{-\al}}{\Ga(1-\al)} (a,\vp)_\LL, \quad \vp\in \HH,
$$
which implies
$$
B_1[ tv(\cdot,t),\vp]=\f{1}{\Ga(1-\al)} (a,\vp)_\LL,\quad \vp\in\HH.
$$
Letting $t\to\infty$, we have
$$
\f{a}{\Ga(1-\al)} = 0 \mbox{ in $\R^d$}.
$$
In this case, the Cauchy problem \eqref{equ-cauchy} has only the trivial solution,
e.g., $u(\cdot,t) = 0$ in $\R^d$ and $t>0$. This completes the proof of Corollary
\ref{coro-asymp}.
\end{proof}
\begin{rem}
The existence of the unique solution $v$ to the equation \eqref{equ-ellip-v} can be easily shown
by using the Lax-Milgram theorem. 
Indeed, we define a bilinear operator $\B_1[\Phi,\Psi]: \HH\times \HH\to\C$
corresponding the equation \eqref{equ-ellip-v} by
$$
\B_1[\Phi,\Psi]
:=\int_{\R^d} \left((-\De)^{\f{\be}{2}}\Phi(x)\right) \Psi(x) - p(x) \Phi(x) \Psi(x) \d x.
$$
Similar to the proof used in Lemma \ref{lem-esti-Lu}, we see that
$$
|\B_1[\Phi,\Psi]|\le C\|\Phi\|_\HH \|\Psi\|_\HH,
$$
moreover, by noting that $p(x)<-\de_0<0$ in $\R^d$, we have
$$
\B_1[\Phi,\Phi] \ge C\|\Phi\|_\HH^2.
$$
Then the Lax-Milgram theorem shows that there is a unique solution $v\in\HH$ to the equation
\eqref{equ-ellip-v} which satisfies
$$
\B_1[v(\cdot,t),\vp] = \f{t^{-\al}}{\Ga(1-\al)} (a,\vp)_\LL,\quad \vp\in \HH.
$$
Thus by taking $\vp=v$, we have
$$
\|v(\cdot,t)\|_\HH^2 \le C\B_1[v(\cdot,t),v(\cdot,t)] \le \f{Ct^{-\al}}{\Ga(1-\al)} |(a,v(\cdot,t))_\LL|.
$$
Hence
$$
\|v(\cdot,t)\|_\HH \le Ct^{-\al} \|a\|_\LL.
$$
\end{rem}

%%%%%%%%%%%%%%%%%%%%%%%%%%%%%%%%%%%%%%%%%%%%%%%%%
%%%%%%%%%%%%%%%%%%%%%%%%%%%%%%%%%%%%%%%%%%%%%%%%%
\section{Asymptotic estimate in a bounded domain}
\label{sec-asymp-bdd}
In this section, we will prove Theorem 1.3.

We first show two useful lemmas.
\begin{lem}
The operator $\AA - p$ is positive and self-adjoint.
\end{lem}
\begin{proof}
Since $\AA$ is positive and self-adjoint, noting that the perturbation by $p$ 
is bounded, we see by Theorem 4.3 (p.287) Kato \cite{K76} that 
$\AA - p$ is self-adjoint.
The positivity follows from $p\le 0$.
\end{proof}
\begin{lem}
\label{lem-equiv}
There exist constants $C_1, C_2 > 0$ such that
$$
C_1\| \AA a\|_\LO \le \| -\AA a + pa\|_{\LO}
\le C_2\| \AA a\|_{\LO}
$$
for all $a \in D(\AA)$.
\end{lem}
\begin{proof}
First we note that
$\| a\|_{\LO} \le C\| \AA a\|_{\LO}$ for $a \in D(\AA)$.  Therefore
$$
\| -\AA a + pa\|_{\LO} \le \| \AA a\|_{\LO} + C\| a\|_{\LO}
\le C\| \AA a\|_{\LO}
$$
for $a \in D(\AA)$, which is the second inequality. Next since $\AA - p$ is positive,
$0 \in \rho(-\AA + p)$: the resolvent. Therefore $(-\AA + p)^{-1}: L^2(\Om) \to L^2(\Om)$
is bounded.  We set $b = (-\AA + p)a$, and $a = (-\AA + p)^{-1}b$.  Then
\begin{align*}
& \| \AA a\|_\LO = \|\AA(-\AA + p)^{-1}b\|_\LO
= \| (\AA - p)(-\AA + p)^{-1}b + p(-\AA + p)^{-1}b\|_\LO
\\
\le& \| b\|_\LO + \| p(-\AA + p)^{-1}b\|_\LO
\le \| b\|_\LO + C\| b\|_\LO,
\end{align*}
that is,
$$
\| \AA a\|_\LO \le (1+C)\| (-\AA + p)a\|_\LO,
$$
which is the first inequality.  Thus the proof of the lemma is
completed.
\end{proof}

In view of Lemma \ref{lem-equiv}, we apply the Heinz-Kato inequality (e.g., Tanabe \cite{Ta79})
and have
$$
D((-\AA+p)^{\te}) = D(A^{\f{\be}{2}\te})
$$
and
$$
C_1\| A^{\f{\be}{2}\te}a\|_\LO
\le \| (-\AA + p)^{\te}a\|_\LO
\le C_2\| A^{\f{\be}{2}\te}a\|_\LO, \quad
a\in D(A^{\f{\be}{2}\te}), \ 0\le \te
\le \f{2}{\be},
$$
that is,
\begin{equation}
\label{esti-equiv-theta}
C_1\| A^{\f{\ga}{2}}a\|_\LO
\le \| (-\AA + p)^{\f{\ga}{\be}}a\|_\LO
\le C_2\| A^{\f{\ga}{2}}a\|_\LO
\end{equation}
and
$$
D((-\AA+p)^{\f{\ga}{\be}}) = D(A^{\f{\ga}{2}}),
\ 0\le \ga \le 2.
$$
Let $\la_n$ be the eigenvalue of $\AA - p$: $0 < \la_1 \le \la_2 \le \la_3 \le \cdots$, where $\la_j$
appears the same times as the multiplicity.  Let $\vp_n$, $n \in \N$, be an orthonormal eigenvector
for $\la_n$.  Then $\{ \vp_n\}_{n\in \N}$ is an orthonormal basis in $\LO$.
By \eqref{esti-equiv-theta}, we have
\begin{equation}
\label{esti-equiv-gamma}
C_1\| a\|_{H^{\ga}(\Om)}
\le \left( \sum_{n=1}^{\infty} \la_n^{\f{2\ga}{\be}}
(a,\vp_n)^2\right)^{\f 1 2}
\le C_2\| a\|_{H^{\ga}(\Om)}
\end{equation}
for $a \in D(A^{\f{\ga}{2}})$ and $0 \le \ga \le 2$.

Now we proceed to 
\begin{proof}[\bf Proof of Theorem \ref{thm-asymp-bdd}]
Similarly to Sakamoto and Yamamoto \cite{SaYa11}, we represent the solution $u$ by
$$
u(t) = \sum_{n=1}^{\infty} (a,\vp_n)E_{\al,1}(-\la_nt^{\al})\vp_n
$$
in $C((0,T]; D(-\AA+p))$.  By \eqref{esti-equiv-gamma}, we estimate $\| u(t)\|_{H^{\ga}(\Om)}$:
\begin{align*}
& \| u(t)\|^2_{H^{\ga}(\Om)}
\le C\sum_{n=1}^{\infty} \la_n^{\f{2\ga}{\be}}
(a,\vp_n)^2  |E_{\al,1}(-\la_nt^{\al})|^2\\
\le& C\sum_{n=1}^{\infty} (a,\vp_n)^2
\left(\f{\la_n^{\f{\ga}{\be}}} {1+\la_n t^{\al}}\right)^2
= Ct^{-\f{2\al\ga}{\be}}
\sum_{n=1}^{\infty} (a,\vp_n)^2
\left( \f{(\la_nt^{\al})^{\f{\ga}{\be}}} {1+\la_n t^{\al}}\right)^2.
\end{align*}

Let $\gamma < \beta$.  Then $\frac{\eta^{\frac{\gamma}{\beta}}}
{1+\eta}$ is monotone decreasing for $\eta > \frac{\gamma}{\beta-\gamma}$.
For sufficiently large $t>0$, we can assume that 
$\lambda_1 t^{\alpha} > \frac{\gamma}{\beta-\gamma}$.
Therefore
$$
\frac{(\lambda_nt^{\alpha})^{\frac{\gamma}{\beta}}}
{1+\lambda_nt^{\alpha}} 
\le \frac{\lambda_1^{\frac{\gamma}{\beta}}t^{\frac{\alpha\gamma}{\beta}}}
{1+\lambda_1t^{\alpha}} 
\le Ct^{\frac{\alpha\gamma}{\beta}}t^{-\alpha},
$$
that is, 
$$
\Vert u(t)\Vert_{H^{\gamma}(\Omega)}^2
\le Ct^{-2\alpha}\sum_{n=1}^\infty (a,\vp_n)^2
$$
for large $t>0$.  For $\gamma=\beta$, since $\frac{\lambda_nt^{\alpha}}
{1+\lambda_nt^{\alpha}} \le 1$, we can directly see the conclusion of the 
theorem.  Thus the proof of the theorem is completed.
\end{proof}

\begin{proof}[\bf Proof of \eqref{notin}]
Let $\ga > \be$.  By Theorem 1.4 (pp.33-34) in Podlubny \cite{P99}, we have
$$
E_{\al,1}(-\la_nt^{\al})
= \f{1}{\la_nt^\al}\f{1}{\Ga(1-\al)} + O\left( \f{1}{\la_n^2t^{2\al}}\right).
$$
We set $\f{\ga}{\be} = 1 + \de$ with small $\de>0$ and set
$u_N(t) = \sum_{n=1}^N (a,\vp_n)E_{\al,1}(-\la_nt^{\al})\vp_n$.
Then
\begin{align*}
&\| u_N(t)\|_{H^{\ga}(\Om)}^2
= \sum_{n=1}^N \la_n^{2+2\de}(a,\vp_n)^2
|E_{\al,1}(-\la_nt^{\al})|^2\\
=& \sum_{n=1}^N \la_n^{2+2\de}(a,\vp_n)^2
\left| \f{1}{\la_nt^\al}\f{1}{\Ga(1-\al)}
+ O\left( \f{1}{\la_n^2t^{2\al}}\right)\right|^2\\
=& \sum_{n=1}^N \la_n^{2+2\de}(a,\vp_n)^2
\left( \left(\f{1}{\la_nt^{\al}}\f{1}{\Ga(1-\al)}
\right)^2
+ O\left( \f{1}{\la_n^3t^{3\al}}\right)\right)\\
\ge & \left(\sum_{n=1}^N \la_n^{2\de}(a,\vp_n)^2\right)
t^{-2\al}\left( \f{1}{\Ga(1-\al)}\right)^2
-  C\sum_{n=1}^N \la_n^{2\de-1}(a,\vp_n)^2t^{-3\al}.
\end{align*}
Here we note
$$
\sum_{n=1}^N \la_n^{2\de-1}(a,\vp_n)^2t^{-3\al}
\le \sum_{n=1}^\infty \la_n^{2\de-1}(a,\vp_n)^2t^{-3\al}
< \infty
$$
because $\de>0$ is small so that $2\de-1<0$.
Thus we verified that $a \not\in H^{2\de}(\Om)$
implies $\lim_{N\to\infty} \sum_{n=1}^N (a,\vp_n)^2\la_n^{2\de}
= \infty$ and so $u(t) \not\in H^{\ga}(\Om)$.
\end{proof}

%%%%%%%%%%%%%%%%%%%%%%%%%%%%%%%%%%%%%%%%%%%%%%%%%
%%%%%%%%%%%%%%%%%%%%%%%%%%%%%%%%%%%%%%%%%%%%%%%%%
\section{Conclusions and remarks}
\label{sec-rem}

In Lemma \ref{lem-analy-u}, we proved that the solution
$u(\cdot,t): (0,T)\to H^\ga(\R^d)$ $(0\le\ga<\be)$ is analytic in time $t\in(0,T)$
and that the short-time estimate $\|u(\cdot,t)\|_{H^\ga(\R^d)}\le C\e^{CT}t^{-\f{\ga}{\be}\al}\|a\|_\LL$
as $t\to0$. Moreover, $u:(0,T)\to H^\ga(\R^d)$ can be analytically extended to the sector
$S_{\te_0}=\{z\in\C\setminus\{0\}: |\arg z|<\te_0\}$ with $\te_0$ such that $\f{\pi}{2}<\te_0<\min\{\pi,\tfrac{\pi}{2\al}\}$.
However, we do not know whether the solution can achieve the sharp spatial regularity $H^{\be}(\R^d)$.

\bigskip

\noindent In Theorem \ref{thm-asymp}, we discuss the long-time asymptotic behavior of the solution
to the Cauchy problem \eqref{equ-cauchy}. We show that the decay rate of the solution is dominated
by $t^{-\f{\al}{2}}$ in the case of $d\le\be$ and by $t^{-\al}$ for $d>\be$ under the assumption that the potential $p$ is in the space $L^\infty(\R^d)$ and
such that $p\le0$ in $\R^d$.
However, from \eqref{def-S} and the property of the Mittag-Leffler function, it is not very difficult to see that the solution $u$ to the Cauchy problem \eqref{equ-cauchy} admits
the sharp estimate
$$
\|u(\cdot,t)\|_{\HHH} \le Ct^{-\f{\al}{2}} \|a\|_\LL,\quad d\ge1,\ 0<\be\le2,
\mbox{ as $t\to\infty$}
$$
and
$$
\|u(\cdot,t)\|_{\HHH} \le Ct^{-\al} (\|a\|_{L^1(\R^d)}+\|a\|_\LL),\quad d>\be,
\mbox{ as $t\to\infty$}
$$
provided that $p\equiv0$ in $\R^d$, that is the estimate in
Theorem \ref{thm-asymp}. However, we do not intend to discuss
whether the first result in Theorem \ref{thm-asymp} is sharp
in the case of $p\le0$. If we assume $p<-\de_0<0$, the decay rate
is shown to be $t^{-\al}$ and sharp.

\bigskip

\noindent In Theorem \ref{thm-asymp-bdd}, the long-time asymptotic behavior of the solution
to the initial-boundary value problem is investigated. By the property of the Mittag-Leffler function
and the eigenfunction expansion, the solution to (1.5) decays as $t^{-\al}$
which is different from the case of unbounded domain, where in the case of $\Om=\R^d$
the decay rate is shown to be dependent of the relation between $\be$ and $d$.

\section*{Acknowledgement}
The research of Xing Cheng is supported in part by ``the Fundamental Research Funds for
the Central Universities"(No.2014B14214). Zhiyuan Li thanks the Leading Graduate Course for
Frontiers of Mathematical Sciences and Physics (FMSP, The University of Tokyo).
The works was partly supported by A3 Foresight Programs \lq \lq Modeling and Computation of
Applied Inverse Problems\rq\rq\  by Japan Society of the Promotion of Science.

%%%%%%%%%%%%%%%%%%%%%%%%%%%%%%%%%%%%%%%%%%%%%%%%%
%%%%%%%%%%%%%%%%%%%%%%%%%%%%%%%%%%%%%%%%%%%%%%%%%

\end{document}